\documentclass{amsart}
\usepackage[mathscr]{euscript}
\usepackage{amsopn}
\usepackage{cite}
\usepackage{hyperref,amsmath,amssymb,amsthm, mathtools, tikz,enumerate}
\theoremstyle{plain}
\newtheorem{thm}{Theorem}[section]
\newtheorem{prop}[thm]{Proposition}
\newtheorem{lem}[thm]{Lemma}
\newtheorem{cor}[thm]{Corollary}
\newtheorem{defn}[thm]{Definition}

\newtheorem{exmp}[thm]{Example}
\newtheorem{notation}[thm]{Notation}

\newcommand{\R}{\mathbb{R}}
\newcommand{\Z}{\mathbb{Z}}

\newcommand{\N}{\mathbb{N}}
\newcommand{\dom}{\text{dom}\;} 
\newcommand{\diag}{\text{diag}} 
\newcommand{\norm}[1]{\left\lVert#1\right\rVert} 
\newcommand{\abs}[1]{\left\lvert #1 \right\rvert} 
\newcommand{\lp}{\left(}
\newcommand{\rp}{\right)}
\newcommand{\lan}{\left\langle}
\newcommand{\ran}{\right\rangle}

\begin{document}
\title{Kernel Stabilization of Unbounded Derivations on $C^*$-algebras}
\author{Lara Ismert}
\address{Department of Mathematics, University of Nebraska-Lincoln, 1400 R Street, Lincoln, NE 68588}
\curraddr{}
\email{lara.ismert@huskers.unl.edu}

  \begin{abstract}
\noindent A derivation $\delta$ on a $C^*$-algebra has \textit{kernel stabilization} if for all $n\in \N$, $\ker \delta^n=\ker \delta.$ Our main result shows that a weakly-defined derivation studied recently by E. Christensen has kernel stabilization. As corollaries, we (1) show that a family of $*$-derivations on $C^*$-algebras studied by Bratteli and Robinson has kernel stabilization and (2)  provide sufficient conditions for when operators satisfying the Heisenberg Commutation Relation must both be unbounded. 
\end{abstract}
\maketitle
\vspace{-1cm}
\section{\label{sec:level1}Introduction}

\indent Given an algebra $\mathscr{A}$ with involution and a fixed element $a\in \mathscr{A}$ such that $a=a^*$, the map $\delta_a:\mathscr{A}\to \mathscr{A}$ by $\delta_a(b):=[ia,b]$ (where $[x,y]=xy-yx$) is a $*$-derivation, that is, $\delta_a(b^*)=\delta_a(b)^*$ for all $b\in \mathscr{A}$. Conversely, for an arbitrary $*$-derivation $\delta:\mathscr{A}\to \mathscr{A}$, certain conditions on the algebra can imply $\delta=\delta_a$ for some $a\in \mathscr{A}.$ The correspondence between derivations on algebras and their representation as commutators has a rich history and is deeply connected to the mathematical formulation of quantum mechanics.

We focus on two settings. The first is when $\mathscr{A}=B(H)$, the set of bounded linear operators on a Hilbert space $H$, and we examine a $*$-derivation on $B(H)$ defined by commutation with an element that is affiliated to a subalgebra of $B(H)$. Specifically, we consider commutators of elements of $B(H)$ with a fixed (possibly unbounded) self-adjoint operator $D$. In the second setting, we consider a $C^*$-algebra $\mathscr{A}$ with a $*$-derivation $\delta$, where the domain of $\delta$ is potentially a proper subspace of $\mathscr{A}$. Under certain conditions, as in Theorem 4 of \cite{Bratteli-Robinson}, $\delta$ and $\mathscr{A}$ may be faithfully represented as commutation with a self-adjoint element, thus returning to the first setting.

The domain of the $*$-derivation in both of these cases is potentially a proper subspace of the algebra. This creates complexities that are not found with derivations defined on the entire $C^*$-algebra. In \cite{Kadison}, Kadison summarizes three of the many significant results pertaining to bounded derivations, which we list below:

\begin{enumerate}
\item Every such derivation on a commutative $C^*$-algebra is 0. (This follows from the Singer-Wermer Theorem from 1955 in \cite{Singer}.)

\item Sakai (1959) showed in \cite{Sakai1} that every derivation on a $C^*$-algebra is automatically bounded, thus affirmatively settling a 1953 conjecture of Kaplansky.

\item In \cite{Kaplansky}, Kaplansky showed every bounded derivation $\delta$ of a type $I$ von Neumann algebra $M$ is \textit{inner}, i.e., there exists $a\in M$ such that $\delta=\delta_a$. 
\end{enumerate}

We turn our attention to densely defined derivations on $C^*$-algebras. Bratteli and Robinson show in \cite{Bratteli-Robinson} that a certain class of unbounded $*$-derivations on $C^*$-algebras can be represented by commutation with an essentially self-adjoint operator $S$. Much more recently, Kadison and Z. Liu have studied unbounded analogues of the aforementioned theorems using Murray-von Neumann algebras in \cite{Kadison-Liu} .

 Let $D$ be an unbounded self-adjoint operator on $H$. Seeking to formalize the connection between commutators and unbounded derivations on $B(H)$ of the form $\delta_D$, Christensen showed in \cite{Christensen1} that $x\in B(H)$ makes $[D, x]$ defined and bounded on a core for $D$ if and only if for every $h,k\in H$, the map $t\mapsto \lan e^{itD}xe^{-itD}h,k\ran$ is continuously differentiable. If $x$ satisfies this, we say $x$ is weakly $D$-differentiable, denoted $x\in \dom \delta^D_w,$ and we define $\delta^D_w(x)$ to be the bounded extension of $[iD,x]$ to all of $H$. Our main result, Theorem\autoref{thm:kerD}, states $\delta^D_w$ has \textit{kernel stabilization}, that is, for every $n\in \N$, $$\ker (\delta^D_w)^n=\ker \delta^D_w.$$
We give two applications of our main result. The first application extends the property of kernel stabilization to a class of unbounded $*$-derivations on $C^*$-algebras considered by Bratteli and Robinson in \cite{Bratteli-Robinson}. This class of derivations is described in the following theorem.

\begin{thm}[Bratteli-Robinson, Theorem 4 \cite{Bratteli-Robinson}]\label{B-R} Let $\delta$ be a derivation of a $C^*$-algebra $\mathscr{A},$ and assume there exists a state $\omega$ on $\mathscr{A}$ which generates a faithful cyclic representation $(\pi, H,f)$ satisfying $$\omega(\delta(a))=0,\;\;\;\forall a\in \dom \delta.$$
Then $\delta$ is closable and there exists a symmetric operator $S$ on $H$ such that $$\dom S=\{h\in H:h=\pi(a)f\;\;\text{ for some }a\in \mathscr{A}\}$$ and $\pi(\delta(a))h=[S,\pi(a)]h,$ for all $a\in \dom \delta$ and all $h\in \dom S.$ Moreover, if the set $\mathscr{A}^\infty$ of analytic vectors for $\delta$ is dense in $\mathscr{A}$, then $S$ is essentially self-adjoint on $\dom S.$ For $x\in B(H)$ and $t\in \R$, define $$\alpha_t(x):=e^{i\overline{S}t}xe^{-i\overline{S}t}$$ where $\overline{S}$ denotes the self-adjoint closure of $S.$ It follows that $\alpha_t(\mathscr{A})=\mathscr{A}$ for all $t\in \R,$ and $\{\alpha_t\}_{t\in \R}$ is a strongly continuous group of automorphisms with closed infinitesimal generator $\widetilde{\delta}$ equaling the closure of $\pi\circ \delta|_{\mathscr{A}^\infty}.$
\end{thm}

Physically, we interpret $\omega$ as a \textit{mixed state} of the quantum system whose observables lie in $\mathscr{A}$. Also, we interpret the condition $\omega(\delta(x))=0$ for all $x\in \dom \delta$ as saying $\omega$ is an \textit{equilibrium state} for the system. For more details, see the introduction of \cite{Bratteli-RobinsonQSM}. We state our application formally below.\\

\noindent \textbf{Application 1.} Let $\mathscr{A}$ be a $C^*$-algebra, $\delta$ a derivation on $\mathscr{A}$, and $\omega$ a state on $\mathscr{A}$ which satisfy the hypotheses of Theorem\autoref{B-R}. For every natural number $n,$ $\ker \delta^n=\ker \delta.$\\

As a second application of our main result, we provide sufficient conditions for when two operators satisfying the Heisenberg Commutation Relation must both be unbounded. 

\begin{defn}\label{HCR} \normalfont Let $A$ and $B$ be two (possibly unbounded) self-adjoint operators on a Hilbert space $H$, with domains $\dom A$ and $\dom B$, respectively. We say $A$ and $B$ \textit{satisfy the Heisenberg Commutation Relation} if there is a dense subspace $K$ of $H$ satisfying $$K\subseteq \dom [A,B]:=\{h\in \dom A\cap \dom B : Ah\in \dom B, Bh\in \dom A\}$$ and $[A,B]k=ik$ for all $k\in K.$
\end{defn}

The classical example of such a pair is the \textit{Schr\"odinger pair}, which we define in Example\autoref{Spair}. Both operators are in this pair are unbounded. A large body of research has been committed to finding sufficient conditions for when two operators satisying the Heisenberg Commutation Relation must be unitarily equivalent to a direct sum of copies of the Schr\"odinger pair. One of the most famous results is the following:

\begin{thm}[Dixmier, \cite{Dixmier}] Suppose $A$ and $B$ are closed symmetric operators on a Hilbert space $H$, and $K$ is a dense subspace of $H$ that is contained in $\dom A\cap \dom B$ and is invariant under $A$ and $B$. If $A$ and $B$ satisfy the Heisenberg Commutation Relation on $K$ and the restriction of $A^2+B^2$ to $K$ is essentially self-adjoint, then $A$ and $B$ are unitarily equivalent to a direct sum of copies of the Schr\"odinger pair. 
\end{thm}

If two operators are unitarily equivalent to copies of the Schr\"odinger pair, then they too must be unbounded. However, there exist examples, one of which we provide in Example\autoref{bddunbdd}, of two operators satisfying the Heisenberg Commutation Relation where one is bounded and the other is unbounded. Our result yields sufficient conditions for when two operators satisfying the Heisenberg Commutation Relation must both be unbounded without showing the two operators are unitarily equivalent to copies of the Schr\"odinger pair.\\

\noindent \textbf{Application 2.} Let $A$ and $B$ be self-adjoint operators on a Hilbert space $H$. If $A$ and $B$ satisfy the Heisenberg Commutation Relation, then (1) $\dom [A,B]$ is not a core for $A$ and $B$ \textit{or} (2) both $A$ and $B$ must be unbounded.\\

Here we outline the rest of the paper. Section 2 is devoted to providing background and summarizing some of Christensen's results from \cite{Christensen1} and \cite{Christensen2}. In Section 3, we prove our main result, and in Section 4, we prove our applications.

\section{\label{sec:level1} Background and Examples of Weak $D$-differentiability}

Let $D$ be a self-adjoint operator with domain $\dom D\subseteq H$. For any $t\in \mathbb{R}$, the operator $e^{itD}$ is unitary, and the one-parameter family $\{e^{itD}\}_{t\in \mathbb{R}}$ is strongly continuous. For $x\in B(H)$ and $t\in \mathbb{R}$, define $\alpha_t(x):=e^{itD}xe^{-itD}$. Then $\{\alpha_t\}_{t\in \mathbb{R}}$ defines a flow on $B(H)$, and more specifically, is a one-parameter automorphism group on $B(H).$ While the \textit{infinitesimal generator} of this automorphism group is a natural derivation on $B(H)$ to consider, we focus instead on a related derivation with a larger domain. 

\begin{defn}\normalfont An operator $x\in B(H)$ is \textit{weakly $D$-differentiable} if there exists $y\in B(H)$ such that for every $h,k\in H,$ $$\lim\limits_{t\to 0} \abs{\lan \lp \frac{\alpha_t(x)-x}{t}-y\rp h,k\ran}=0.$$ Equivalently, for every $h,k\in H$ the function $t\mapsto \lan \alpha_t(x)h,k\ran$ is continuously differentiable.
\end{defn}

\begin{thm}[Christensen, 3.8 \cite{Christensen1}]\label{weaklyd} Let $x$ be a bounded operator on $H$. The following properties are equivalent:
\begin{enumerate}[(i)]
\item $x$ is weakly $D$-differentiable.
\item There exists $y\in B(H)$ such that for every $h\in H,$ $$\lim\limits_{t\to 0} \norm{ \lp \frac{\alpha_t(x)-x}{t}-y\rp h}=0.$$
\item There exists $c>0$ such that for all $t\in \mathbb{R}$, $$\norm{\alpha_t(x)-x}\leq c\abs{t}.$$
\item The commutator $[iD,x]$ is defined and bounded on the domain of $D.$
\item The commutator $[iD,x]$ is defined and bounded on a core for $D.$
\item The sesquilinear form on $\dom D\times \dom D$ given by $$(h,k)\mapsto i\lan xh,Dk\ran-i\lan xDh,k\ran$$ is bounded.
\item The matrix $m([iD,x])_{rc}=i(DP_rxP_c-P_rxP_cD)$ defines a bounded operator on $H$, where $(P_n)_{n\in \mathbb{Z}}$ are the spectral projections of the intervals $(n-1,n]$.
\end{enumerate}
If any of the above conditions hold, then $x(\dom D)\subseteq \dom D$ and $\delta^D_w(x)|_{\dom D}=i[D,x].$ We write $x\in \dom \delta^D_w$ and the $y$ in item $(ii)$ satisfies $y=\delta^D_w(x)$. Moreover, for any $h,k\in H,$ $\frac{d}{dt}\lan \alpha_t(x)h,k\ran=\lan \alpha_t(\delta^D_w(x))h,k\ran.$
\end{thm}

\begin{thm}[Christensen, 3.9\cite{Christensen1}]\label{propertiesofdelta_D} The domain of definition $\dom \delta^D_w$ is a strongly dense $*$-subalgebra of $B(H)$ and $\delta^D_w$ is a $*$-derivation into $B(H)$. The graph of $\delta^D_w$ is weak operator topology closed.
\end{thm}

\begin{defn}\normalfont An operator $x\in B(H)$ is \textit{$n$-times weakly $D$-differentiable} if for every $k=0,...,n-1$, $(\delta^D_w)^{k}(x)\in \dom \delta^D_w$. We denote this by $x\in \dom(\delta^D_w)^n.$
\end{defn}

\begin{prop}[Christensen, 2.6 \cite{Christensen2}] A bounded operator $x$ on $H$ is $n$-times weakly $D$-differentiable if and only if for any pair $h,k\in H$ the function $t\mapsto \lan \alpha_t(x)h,k\ran$ is $n$-times continuously differentiable. If $x$ is $n$-times weakly $D$-differentiable, then $$\frac{d^n}{dt^n}\lan \alpha_t(x)h,k\ran=\lan \alpha_t((\delta^D_w)^n(x))h,k\ran.$$
\end{prop}

Analogous to Theorem \autoref{weaklyd}, Christensen shows in \cite{Christensen2} that higher order weak $D$-differentiability is directly tied to iterated commutators $[iD,...,[iD,x]]$. 

\begin{prop}[Christensen, 3.3 \cite{Christensen2}]\label{prop3.3} Let $x\in \dom (\delta^D_w)^n$. Then for $k=1,...,n,$
\begin{enumerate}[(i)]
\item $(\delta^D_w)^{k-1}(x) (\dom D)\subseteq \dom D$
\item $x (\dom D^k)\subseteq \dom D^k$
\item $\dom \underbrace{[iD,...,[iD,x]]}_{k\text{ times}}=\dom D^k$
\item $(\delta^D_w)^k(x)|_{\dom D^k}=\underbrace{[iD,...,[iD,x]]}_{k\text{ times}}$
\item $(\delta^D_w)^k(x)$ is the bounded extension of $\underbrace{[iD,...,[iD,x]]}_{k\text{ times}}$ from $\dom D^k$ to all of $H$.
\end{enumerate}
\end{prop}

\begin{thm}[Christensen, 4.1 \cite{Christensen2}]\label{ntimes} Let $x\in B(H)$ and $n$ be a natural number. The following are equivalent: 
\begin{enumerate}[(i)]
\item $x\in \dom(\delta^D_w)^n$.
\item $x$ is $n$ times weakly $D$-differentiable.
\item For all $k=1,...,n,$ $x(\dom D^k)\subseteq \dom D^k$ and $\underbrace{[iD,...,[iD,x]]}_{k\text{ times}}$ is defined and bounded on $\dom D^k$ with closure $(\delta^D_w)^k(x)$.
\item There exists a core $\mathscr{X}$ for $D$ such that for any $k=1,...,n,$ the operator $\underbrace{[iD,...,[iD,x]]}_{k\text{ times}}$ is defined and bounded on $\mathscr{X}.$
\end{enumerate}
\end{thm}

\begin{notation} For notational convenience, we define $$d^k(x):=\underbrace{[iD,...,[iD,x]]}_{k\text{ times}}$$ for each $k\in \N.$
\end{notation}

We now present the motivating example for Theorem \autoref{thm:kerD}. Given a $\sigma$-finite measure space $(X,\mu),$ define $$\diag:L^\infty(X,\mu)\to B(L^2(X,\mu))$$ $$\diag(f):=M_f,$$ where $M_fg=fg$ for each $g\in L^2(X,\mu).$ Throughout, we denote the standard orthonormal basis for $\ell^2(\Z)$ by $\{\epsilon_j:j\in \Z\},$ and we denote the matrix representation of an operator $x\in B(\ell^2(\Z))$ with respect to the standard orthonormal basis by $[x_{rc}]$ where $$x_{rc}:=\lan x\epsilon_c,\epsilon_r\ran.$$

\begin{exmp}\label{l^2(Z)}\normalfont  Define $(Df)(j):=jf(j)$ for $f\in \dom D,$ where $$\dom D:=\{f\in \ell^2(\Z): \sum_{j\in \Z} j^2\abs{f(j)}^2<\infty\}.$$ Then,
\begin{enumerate}[(a)]
\item the operator $D$ is self-adjoint.
\item an operator $x\in B(\ell^2(\Z))$ is $n$-times weakly $D$-differentiable if and only if for every $k\leq n$, $x(\dom D^k)\subseteq \dom D^k$ and the matrix $[i^k(r-c)^kx_{rc}]$ with dense domain $\dom D^k$ extends to a bounded operator on $\ell^2(\Z)$. When either condition is satisfied, $$[(\delta^D_w)^n(x)_{rc}]|_{\dom D^n}=[i^n(r-c)^nx_{rc}].$$

\item for any $g\in \ell^\infty(\Z)$, $\delta^D_w(M_g)=0.$
\item for all $n\in \N$, $\ker (\delta^D_w)^n=\diag(\ell^\infty(\Z))$.

\end{enumerate}
\end{exmp}
\begin{proof} 
\begin{enumerate}[(a)]
\item See Example 7.1.5 of \cite{Simon}.
\item Matrix multiplication shows for any $r,c\in \Z,$ $$d^k(x)_{rc}=i^k(r-c)^kx_{rc}.$$ Given $x\in B(\ell^2(\Z))$ such that $x(\dom D^k)\subseteq \dom D^k$ for each $k\leq n$, the domain of $d^k(x)$ is $\dom D^k.$ Theorem\autoref{ntimes} states $x$ is $n$-times weakly $D$-differentiable if and only if for every $k\leq n,$ $x(\dom D^k)\subseteq \dom D^k$ and $d^k(x)$ is bounded on $\dom D^k.$ It follows that $x$ is $n$-times weakly $D$-differentiable if and only if $x(\dom D^k)\subseteq \dom D^k$ and $[d^k(x)_{rc}]=[i^k(r-c)^kx_{rc}]$ is bounded on $\dom D^k$. As $D$ is self-adjoint, $\dom D^k$ is dense in $\ell^2(\Z)$ for all $k\in \N$. Therefore, $[d^k(x)_{rc}]$ extends to a bounded matrix on all of $\ell^2(\Z).$ By Theorem\autoref{ntimes}, the closure $(\delta^D_w)^n(x)$ is the extension of $[i^n(r-c)^nx_{rc}]$ to all of $\ell^2(\Z).$

\item Fix $g\in \ell^\infty(\Z),$ and let $f\in \dom D.$ We show $M_gf\in \dom D.$ Observe $$\sum_{j\in \Z} \abs{j(M_gf)(j)}^2=\sum_{j\in \Z} \abs{jg(j)f(j)}^2\leq \norm{g}^2_\infty \lp \sum_{j\in \Z} \abs{jf(j)}^2\rp<\infty.$$ As $f\in \dom D$ was arbitrary, $M_g(\dom D)\subseteq \dom D,$ and hence, the commutator $[iD,M_g]$ is a well-defined linear operator on $\dom D.$ Furthermore, $iD$ and $M_g$ are diagonal matrices with complex entries (which commute), so the commutator $[iD,M_g]$ is simply a restriction of the 0 operator to $\dom D.$ Theorem\autoref{weaklyd} implies $M_g\in \dom \delta^D_w$ and $\delta^D_w(M_g)$ is the extension of $[iD,M_g]$ to all of $H$. In particular, $\delta^D_w(M_g)=0.$ Hence, $M_g\in \ker \delta^D_w$, and since $g\in \ell^\infty(\Z)$ was arbitrary, $\diag(\ell^\infty(\Z))\subseteq \ker \delta^D_w.$

\item Part (c) quickly implies diag$(\ell^\infty(\Z))\subseteq \ker(\delta^D_w)^n$ for all $n\in \N.$ We now show if $(\delta^D_w)^n(x)=0$, then $x\in\diag(\ell^\infty(\Z)).$ If $x\in \dom (\delta^D_w)^n$ and $(\delta^D_w)^n(x)=0$, then $x\in B(\ell^2(\Z))$ and $(\delta^D_w)^n(x)_{rc}=0$ for every $r,c\in \Z.$ By part (b), $[(\delta^D_w)^n(x)_{rc}]|_{\dom D^n}=[i^n(r-c)^nx_{rc}],$ thus, $i^n(r-c)^nx_{rc}=0$ for every $r,c\in \Z.$ If $r\neq c,$ it must be that $x_{rc}=0$, i.e., $x$ must be zero off the diagonal. As $x\in B(\ell^2(\Z))$, we conclude $x\in \diag(\ell^\infty(\Z)).$ Therefore, $\ker(\delta^D_w)^n=\diag(\ell^\infty(\Z))$ for all $n\in \N.$
\end{enumerate}
\end{proof}

This kernel stabilization phenomenon initially appears unique to the setting of Example\autoref{l^2(Z)}; the self-adjoint operator is multiplicity-free (the von Neumann algebra generated by its spectral projections is a maximal abelian self-adjoint subalgebra of $B(\ell^2(\Z))$) and its eigenvectors constitute our choice of orthonormal basis. In Section 3, we show our example is not unique; kernel stabilization holds for every self-adjoint operator on any Hilbert space.

\section{\label{sec:level1}Kernel Stabilization of $\delta^D_w$}

In this section, we show for any self-adjoint operator $D$ on a Hilbert space, $\ker (\delta^D_w)^n=\ker \delta^D_w$ for all $n\in \N.$ We call this property \textit{kernel stabilization}.

 \begin{prop}\label{vN} Let $H$ be a Hilbert space and $D$ a self-adjoint operator. Then $\ker \delta^D_w$ is a von Neumann algebra.
\end{prop}
\begin{proof} The identity $I$ of $B(H)$ is easily shown to be in $\ker \delta^D_w.$ Let $x\in \ker \delta^D_w.$ As $\dom \delta^D_w$ is a $*$-algebra by Theorem\autoref{propertiesofdelta_D}, $x^*\in \dom \delta^D_w.$ Since $\delta^D_w$ is a $*$-derivation, $\delta^D_w(x^*)=\delta^D_w(x)^*=0.$ Therefore, $x^*\in \ker\delta^D_w.$ Finally, if $x,y\in \ker \delta^D_w$, then $xy\in\dom\delta^D_w$ and $\delta^D_w(xy)=\delta^D_w(x)y+x\delta^D_w(y)=0$, so $xy\in \ker\delta^D_w.$

Let $(x_\lambda)\subset \ker \delta^D_w$ be a net converging in the weak operator topology to some $x\in B(H).$ We show $x\in \dom \delta^D_w$ and $\delta^D_w(x)=0.$ Because $\delta^D_w(x_\lambda)=0$ for all $\lambda,$ we trivially have $\delta^D_w(x_\lambda)\overset{\text{WOT}}{\to} 0$.  By Theorem\autoref{propertiesofdelta_D}, the graph of $\delta^D_w$ is weak operator topology closed. Therefore, $x\in \dom \delta^D_w$ and $\delta^D_w(x)=0.$ We conclude $\ker \delta^D_w$ is a von Neumann algebra. 
\end{proof}

\begin{notation} \normalfont Let $\mathscr{P}_D$ denote the collection of all spectral projections for $D$ obtained through the spectral theorem for unbounded self-adjoint operators. Also, let $$\mathscr{M}_D:=\mathscr{P}_D''.$$ 
\end{notation}
We give further description of the structure $\ker \delta^D_w$ in terms of of $\mathscr{M}_D$ in the following lemma and proposition.
\begin{lem}\label{lem:projectioncommutator} Suppose $x\in B(H)$ satisfies $x(\dom D)\subseteq \dom D.$ If $P\in \mathscr{P}_D$, then $$[P,[D,x]]h=[D,[P,x]]h$$ for all $h\in \dom D.$
\end{lem}
\begin{proof} Let $B(\R)$ denote the bounded Borel functions on $\R$, and for each $R\in \R$, define id$_R:\R\to \R$ by id$_R(t)=t$ whenever $-R\leq t\leq R$ and id$_R(t)=0$ otherwise. The spectral theorem, stated as in Theorem 7.2.8 \cite{Simon}, provides a bounded Borel functional calculus for $D$, that is, a $*$-homomorphism $\Phi_D:B(\R)\to B(H)$ satisfying $\Phi_D(1)=I$, $$\dom D=\{h\in H : \lim\limits_{R\to \infty}\norm{\Phi_D(\text{id}_R)h}<\infty\},$$ and $$Dh=\lim\limits_{R\to \infty}\Phi_D(\text{id}_R)h$$ for all $h\in \dom D.$ We claim for each $P\in \mathscr{P}_D$, $P(\dom D)\subseteq \dom D$ and $PDh=DPh$ for all $h\in \dom D$. Given $P\in \mathscr{P}_D$, $P=\Phi_D(\chi_E)$ for some Borel set $E\subseteq \R.$ Note that (id$_R\cdot \chi_E)(t)=0$ if $t\not\in E\cap [-R,R]$, and otherwise (id$_R\cdot \chi_E)(t)=t.$ Thus, for any $h\in \dom D$, $$\lim\limits_{R\to\infty}\norm{\Phi_D(\text{id}_R)Ph}=\lim\limits_{R\to\infty}\norm{\Phi_D(\text{id}_R)\Phi_D(\chi_E)h}=\lim\limits_{R\to\infty}\norm{\Phi_D(\text{id}_R\cdot \chi_E)h}\leq \lim\limits_{R\to\infty}\norm{\Phi_D(\text{id}_R)h}<\infty.$$Therefore, $Ph\in \dom D$, and as $h\in \dom D$ was arbitrary, $P(\dom D)\subseteq \dom D$. Furthermore,
\begin{align*}
\norm{DPh-PDh}
&=\lim\limits_{R\to\infty}\norm{ \Phi_D(\text{id}_R)\Phi_D(\chi_E)h-\Phi_D(\chi_E)\Phi_D(\text{id}_R)h}\\
&=\lim\limits_{R\to\infty}\norm{ \Phi_D(\text{id}_R\cdot \chi_E)h-\Phi_D(\chi_E\cdot \text{id}_R)h}\\
&=\lim\limits_{R\to\infty}\norm{ \Phi_D(\text{id}_R\cdot \chi_E)h-\Phi_D(\text{id}_R\cdot\chi_E)h}\\
&=0.
\end{align*}
Given $x(\dom D)\subseteq \dom D$, for any $h\in \dom D$ we observe
\begin{align*} 
[P,[D,x]]h
&=P(Dx-xD)h-(Dx-xD)Ph\\
&=PDxh-PxDh-DxPh+xDPh\\
&=DPxh-PxDh-DxPh+xPDh\\
&=DPxh-DxPh+xPDh-PxDh\\
&=D(Px-xP)h+(xP-Px)Dh\\
&=D(Px-xP)h-(Px-xP)Dh\\
&=[D,[P,x]]h
\end{align*} Hence, $[P,[D,x]]h=[D,[P,x]]h$ for all $h\in \dom D$, and as $P\in \mathscr{P}_D$ was arbitrary, this equality holds for any spectral projection of $D$.
\end{proof}

\begin{prop}\label{prop:ker=M_D'} $\mathscr{M}_D\subseteq \ker \delta^D_w= \mathscr{M}_D'.$
\end{prop}
\begin{proof} Let $P\in \mathscr{P}_D.$ By the previous lemma, $[D,P]=0$ on $\dom D$, so $P\in \dom \delta^D_w$ by Theorem\autoref{weaklyd}. Moreover, $\delta^D_w(P)$ is the bounded extension of $i(DP-PD)$ to all of $H$, which is 0. Therefore, $P\in \ker\delta^D_w.$ Proposition\autoref{vN} implies $\mathscr{M}_D\subseteq \ker \delta^D_w$.

 Let $x\in \ker \delta^D_w.$ By Theorem\autoref{ntimes}, $x (\dom D)\subseteq \dom D$ and $\delta^D_w(x)|_{\dom D}=[D,x]|_{\dom D}=0.$ Then, by Theorem X.4.11 \cite{Conway}, $xf(D)\subseteq f(D)x$ for any $f \in B(\R)$. In particular, when $f=\chi_E$ for some Borel subset $E\subseteq \R$ and $P$ denotes the corresponding spectral projection for $D$, $ xP=Px$. Hence, $x$ commutes with all projections in $\mathscr{P}_D$, and as $\mathscr{M}_D$ is generated as a von Neumann algebra by these projections, it follows that $x\in \mathscr{M}_D^{'}.$

Let $x\in \mathscr{M}_D'.$ For each $t\in \R$, $e^{itD}\in \mathscr{M}_D$. Thus, $\alpha_t(x)=e^{itD}xe^{-itD}=x$ for all $t \in \R$. In particular, for any $h,k\in H,$ the function $t\mapsto\lan \alpha_t(x)h,k\ran=\lan xh,k\ran$ is constant, and thus is continuously differentiable with derivative 0. Therefore, $x\in \ker\delta^D_w.$
\end{proof}
We now present our main result.

\begin{thm}\label{thm:kerD} If $D$ is any self-adjoint operator on a Hilbert space $H$, then for every $n\in \N,$ $$\ker (\delta^{D}_w)^n=\ker \delta^{D}_w.$$
\end{thm}
\begin{proof} We first show $\ker (\delta^{D}_w)^2=\ker \delta^{D}_w.$ The inclusion $\ker \delta^{D}_w\subseteq \ker(\delta^D_w)^2$ is clear. Let $x\in \ker(\delta^{D}_w)^2.$ Proposition\autoref{prop:ker=M_D'} states $\ker \delta^{D}_w=\mathscr{M}_D'.$ Thus, it suffices to prove $x\in \mathscr{M}_D',$ which holds if and only if $[P,x]=0$ for every $P\in \mathscr{P}_D.$ By Proposition\autoref{prop3.3}, if $x\in \dom (\delta^D_w)^2$, then $x(\dom D)\subseteq \dom D$, $\delta^D_w(x)(\dom D)\subseteq \dom D$, and $(\delta^D_w)^2(x)|_{\dom D}=[iD,\delta^D_w(x)].$ Since $(\delta^D_w)^2(x)=0,$ it must be that $[iD,\delta^D_w(x)]=0.$ Thus, Theorem X.4.11 of \cite{Conway} implies $\delta^D_w(x)$ commutes with the bounded Borel functional calculus for $D,$ so, in particular, $[P,\delta^D_w(x)]=0$ for every $P\in \mathscr{P}_D.$ Because $\delta^D_w(x)$ and $P$ both preserve the domain of $D,$  so does the commutator $[P,\delta^D_w(x)]$. Thus, Lemma \autoref{lem:projectioncommutator} implies $$0=[P,\delta^D_w(x)]|_{\dom D}=[P,[iD,x]]|_{\dom D}=[iD,[P,x]]|_{\dom D}.$$ As $[P,x]\in B(H)$, $[P,x](\dom D)\subseteq \dom D,$ and $[iD,[P,x]]$ is bounded on the domain of $D,$ Theorem\autoref{ntimes} implies $[P,x]\in \ker\delta^D_w.$ Hence, by Proposition\autoref{prop:ker=M_D'}, $[P,x]\in \mathscr{M}_D'.$ Therefore, 
\begin{align*}
[P,x]
&=(P+P^{\perp})[P,x](P+P^{\perp})\\
&=P[P,x]P+P[P,x]P^{\perp}+P^{\perp}[P,x]P+P^{\perp}[P,x]P^{\perp}\\
&=P[P,x]P+PP^{\perp}[P,x]+P^{\perp}P[P,x]+P^\perp[P,x]P^\perp\\
&=P(Px-xP)P+0+0+P^\perp(Px-xP)P^\perp\\
&=PxP-PxP+0+0+0\\
&=0.
\end{align*}
As $P\in \mathscr{P}_D$ was arbitrary, $x\in \mathscr{M}_D'.$ By Proposition\autoref{prop:ker=M_D'}, $x\in \ker \delta^{D}_w.$

We proceed by induction on $n$. The case when $n=1$ is vacuous. Suppose $\ker(\delta^D_w)^k=\ker \delta^D_w$ for some $k\in \N$. Let $x\in \ker (\delta^{D}_w)^{k+1}.$ Then $\delta^D_w(x)\in \ker (\delta^D_w)^k$, which equals $\ker\delta^D_w$ by the inductive hypothesis. Hence, $x\in \ker(\delta^D_w)^2.$ Since we have already shown $\ker(\delta^D_w)^2=\ker \delta^D_w$, we have $x\in \ker\delta^D_w.$ Therefore, $\ker(\delta^D_w)^n=\ker \delta^D_w$ for all $n\in \N$.
\end{proof}

\section{\label{sec:level1}Applications of Theorem\autoref{thm:kerD}}

The first application of Theorem\autoref{thm:kerD} is in the context of Theorem\autoref{B-R}. We first define the derivation $\delta^D_u$, which is simply the infinitesimal generator $\widetilde{\delta}$ for the one-parameter automorphism group given by $\alpha_t(x):=e^{itD}xe^{-itD}$ for each $t\in \R$. \begin{defn}\label{uniformDdifferentiability} \normalfont An operator $x\in B(H)$ is \textit{uniformly $D$-differentiable} if there exists $y\in B(H)$ such that $$\lim\limits_{t\to 0}\norm{\frac{\alpha_t(x)-x}{t}-y}=0.$$ We denote this by $x\in \dom \delta^D_u$ and $\delta^D_u(x)=y.$
\end{defn}
\vspace{0.3cm}
\begin{prop}\label{keru=kerw} $\ker \delta^D_u=\ker \delta^D_w.$
\begin{proof} Theorem 4.1 \cite{Christensen1} states $x\in \dom \delta^D_u$ if and only if $x\in \dom \delta^D_w$ and $t\mapsto \alpha_t(\delta^D_w(x))$ is norm continuous. Moreover, $\delta^D_w$ extends $\delta^D_u$. Thus, $\ker \delta^D_u\subseteq \ker \delta^D_w$.

Let $x\in \ker \delta^D_w.$ Then $t\mapsto\alpha_t(\delta^D_w(x))=0$ is norm continuous, and hence, $x\in \dom \delta^D_u.$ Moreover, $\delta^D_u(x)=\delta^D_w|_{\dom \delta^D_u}(x)=0.$ Therefore, $x\in \ker \delta^D_u.$
\end{proof}
\end{prop}

\begin{cor}\label{stabilizationofkeru} For all $n\in \N,$ $\ker (\delta^D_u)^n=\ker \delta^D_u$. 
\begin{proof} Fix $n>1$ and let $x\in \ker(\delta^D_u)^n.$ Then $(\delta^D_u)^{n-1}(x)\in \dom \delta^D_u.$ Hence, $(\delta^D_u)^{n-1}(x)\in \dom \delta^D_w.$ Further, as $x\in \dom \delta^D_u,$ we have $x\in \dom \delta^D_w$ and $\delta^D_w(x)=\delta^D_u(x).$ Hence, $x\in \dom (\delta^D_w)^n$ and $(\delta^D_w)^n(x)=(\delta^D_u)^n(x)=0.$ By Theorem\autoref{thm:kerD}, $x\in \ker \delta^D_w.$ By Proposition\autoref{keru=kerw}, $x\in \ker \delta^D_u.$ 
\end{proof}
\end{cor}

Given a self-adjoint operator $D$, our proof of kernel stabilization of $\delta^D_w$ relied on the relationship between $\delta^D_w$ and commutation with $D$. Intuitively, then, kernel stabilization is likely to occur for a derivation $\delta$ on an abstract $C^*$-algebra that can be implemented, under an appropriate representation, as commutation with a self-adjoint operator. Bratteli and Robinson provide sufficient conditions for when a derivation on a $C^*$-algebra has such a representation.

Under this representation $\pi$, Bratteli and Robinson construct an essentially self-adjoint operator $S$ which implements the derivation's action as commutation with $S$. Once this essentially self-adjoint operator is in play, we use its self-adjoint closure $D=\overline{S}$ to generate the weak-$D$ derivation $\delta^D_w.$ We show $\delta^D_w$ extends $\delta\circ \pi$ and apply Theorem\autoref{thm:kerD} (kernel stabilization of $\delta^D_w$) to obtain kernel stabilization of $\delta.$

\begin{defn}\label{infgen} \normalfont Given a one-parameter group $\{\alpha_t\}_{t\in \R}$ of maps on $B(H),$ let $\dom \widetilde{\delta}$ be the set of all $x\in B(H)$ so that there exists $y\in B(H)$ satisfying $$\lim\limits_{t\to 0}\norm{\frac{\alpha_t(x)-x}{t}-y}=0.$$ For $x\in \dom\widetilde{\delta},$ let $\widetilde{\delta}(x)=y$ where $y$ is the uniform limit described above. We call $\widetilde{\delta}$ the \textit{infinitesimal generator} for $\{\alpha_t\}_{t\in \R}$.
\end{defn}
\vspace{0.3cm}
\textit{Remark.} When $\alpha_t(x):=e^{itD}xe^{-itD}$ for some self-adjoint operator $D$, Definition\autoref{infgen} is identical to the derivation $\delta^D_u$ in Definition\autoref{uniformDdifferentiability}. 

\begin{defn}\normalfont Let $T$ be a linear operator on a Banach space $X$. A vector $a\in X$ is an \textit{analytic vector for T} if $$a\in \bigcap\limits_{k\in \N} \dom T^k$$ and there exists $t>0$ such that the following series converges: $$\sum_{k=0}^\infty \frac{t^k}{k!} \norm{T^ka}.$$ We denote the set of analytic vectors for $T$ by $X^\infty_T$ or simply $X^\infty$ when $T$ is clear from context.
\end{defn}

\begin{lem}\label{kerofdelta} If $\delta$, $\mathscr{A}$, $\pi$, and $\widetilde{\delta}$ are as in Theorem\autoref{B-R}, then $$\ker \widetilde{\delta}^n\cap \pi(\mathscr{A}^\infty)=\pi(\ker \delta^n)$$ for all $n\in \N$. 
\end{lem}
\begin{proof} Recall if $a\in \mathscr{A}^\infty,$ then Theorem\autoref{B-R} provides $\widetilde{\delta}(\pi(a))=\pi(\delta(a))$. It follows by analyticity of $a$ that $\widetilde{\delta}^n(\pi(a))=\pi(\delta^n(a))$ for every $n\in \N$. Suppose $\widetilde{\delta}^n(\pi(a))=0.$ Then $\pi(\delta^n(a))=\widetilde{\delta}^n(\pi(a))=0,$ and since $\pi$ is faithful, $\delta^n(a)=0.$ Therefore, $\pi(a)\in \pi(\ker\delta^n).$

Conversely, suppose $a\in \ker\delta^n$. Then $a\in \mathscr{A}^\infty$ because $\delta^j(a)=0$ for all $j\geq n$ and $\sum_{k=0}^\infty \frac{t^k}{k!}\norm{\delta^k(a)}=\sum_{k=0}^{n-1}\frac{t^k}{k!}\norm{\delta^k(a)}<\infty$ for any choice of $t>0.$ Similar to above, $\widetilde{\delta}^n(\pi(a))=\pi(\delta^n(a))=\pi(0)=0.$ Therefore, $\pi(a)\in \ker \widetilde{\delta}^n\cap \pi(\mathscr{A}^\infty).$ The desired equality holds for all $n\in \N.$
\end{proof}

\begin{thm}\label{stabilizationofkerdelta} If $\delta$, $\mathscr{A}$, $\pi$, $\widetilde{\delta}$, and $S$ are as in Theorem\autoref{B-R}, then $\ker \delta^n=\ker \delta.$
\end{thm}
\begin{proof} Fix $n\in \N$, and let $a\in \ker \delta^n$. Then, $a\in \mathscr{A}^\infty$ and $\pi(a)\in \ker \widetilde{\delta}^n$ by Lemma\autoref{kerofdelta}. Note $\widetilde{\delta}=\delta^D_u$ where $D=\overline{S}$, so Proposition\autoref{stabilizationofkeru} implies $\ker\widetilde{\delta}^n=\ker\widetilde{\delta}$ for all $n\in \N$. Hence, $\pi(a)\in \ker \widetilde{\delta}\cap \pi(\mathscr{A}^\infty)$. By another application of Lemma\autoref{kerofdelta}, we get $a\in \ker \delta.$ Therefore, $\ker \delta^n=\ker \delta$ for all $n\in \N$.
\end{proof}

The second application of Theorem\autoref{thm:kerD} is related to the Heisenberg Commutation Relation, defined in Definition\autoref{HCR}.
\begin{exmp}\label{Spair} \normalfont The classical example of a pair satisfying the Heisenberg Commutation Relation is the \textit{Schr\"odinger pair}, the quantum mechanical position operator $Q$ and momentum operator $P$ on $L^2(\R)$. Let $$\dom Q=\{f\in L^2(\R): \int_\R \abs{xf(x)}^2\,dx<\infty\}$$ and, for $g\in \dom Q,$ define $(Qg)(x)=xg(x)$ for a.e. $x\in \R$. It is shown in Example 7.1.5 of \cite{Simon} that $Q$ defines a self-adjoint operator. If a function $f$ is absolutely continuous, denote its almost-everywhere defined derivative by $f'.$ Now, let $$\dom P=\{f\in L^2(\R): f\text{ is absolutely continuous and } f'\in L^2(\R)\},$$ and for $h\in \dom P$, define $Ph:=ih'.$ It is shown in Theorem 6.30 of \cite{Weidmann} that $P$ defines a self-adjoint operator. Let $S(\R)$ denote the Schwartz space on $\R$, that is, $$S(\R)=\left\{f\in C^\infty(\R): \forall m,n\in \N,\; \norm{Q^mP^nf}_\infty<\infty\right\}.$$ Proposition X.6.5 of \cite{Conway} shows $S(\R)$ is dense in $L^2(\R)$, and it is clear from its definition that $S(\R)$ is contained in $\dom Q\cap \dom P$ and is invariant under both $Q$ and $P$. Hence, $S(\R)\subseteq \dom[P,Q].$ Furthermore, $[P,Q]g=ig$ for all $g\in S(\R)$. Therefore, $P$ and $Q$ satisfy the Heisenberg Commutation Relation. 
 \end{exmp}

If two operators are unitarily equivalent to a direct sum of copies of the Schr\"odinger pair, then they are certainly both unbounded. There are, however, examples of operators satisfying the Heisenberg Commutation Relation where one operator is bounded.

\begin{exmp}\label{bddunbdd} \normalfont For $f\in L^2[0,1],$ define $(Bf)(x)=xf(x)$ for a.e. $x\in [0,1].$ In contrast to its unbounded analogue $Q$, the operator $B$ is contractive. Let $AC[0,1]$ denote the set of functions which are absolutely continuous on $[0,1]$, and let $$\dom A=\{f\in AC[0,1]: f' \in L^2[0,1],\;f(0)=f(1)\}.$$ For $g\in \dom A,$ define $Ag=ig'.$ Example X.1.12 of \cite{Conway} shows  the operator $A$ with this particular domain is self-adjoint. Due to boundedness of $B$, $$\dom [A,B]=\{f\in \dom A: Bf\in \dom A\}.$$ Choose $$K:=\{f\in AC[0,1]: f'\in L^2[0,1],\;f(0)=f(1)=0\}.$$ Example X.1.11 of \cite{Conway} shows $K$ is dense in $L^2[0,1]$ as it contains all polynomials $p$ on $[0,1]$ satisfying $p(0)=p(1)=0$. Furthermore, we claim $K$ is invariant for $B$. Indeed, products of absolutely continuous functions are again absolutely continuous, so $(Bg)(x)=xg(x)$ for a.e. $x\in [0,1]$ defines an absolutely continuous function. The a.e.-defined derivative of $Bg$ is equivalent to $Bg'+g$ by the product rule. Moreover, $Bg'+g$ belongs to $L^2(\R)$ as $g'\in L^2(\R)$ and $B\in B(L^2[0,1]).$ Lastly, $$(Bg)(0)=0\cdot g(0)=0=1\cdot 0=1\cdot g(1)=(Bg)(1).$$ Thus, $BK\subseteq K$. As a result, $K\subseteq \dom [A,B]$. For $k\in K,$ observe $$[A,B]k=i\left(\frac{d}{dx}(Bk)-B(k')\right)=i(Bk'+k-Bk')=ik.$$ Therefore, $A$ and $B$ satisfy the Heisenberg Commutation Relation. 
\end{exmp} 

We claim the boundedness of the operators in Examples\autoref{Spair} and \autoref{bddunbdd} differs due the relative size of $\dom [P,Q]$ in $L^2(\R)$ versus $\dom [A,B]$ in $L^2[0,1].$ In particular, $\dom [A,B]$ does not contain a core for $A$ or $B$, while $\dom [P,Q]$ contains $S(\R)$, which is a core for both $P$ and $Q$. 

\begin{prop} Let $H$ be a Hilbert space and $D$ be a self-adjoint operator. If 
\begin{enumerate}[(i)]
\item$x\in B(H)$, 
\item $\dom [D,x]$ contains a core $\mathscr{X}$ for $D$, 
\item $[D,x]$ is bounded on $\mathscr{X}$, and 
\item the continuous extension $y$ of $[D,x]|_\mathscr{X}$ to all of $H$ belongs to $\mathscr{M}_D',$ 
\end{enumerate}
then $y=0.$ 
\end{prop}
\begin{proof} Let $x\in B(H)$ and suppose $\dom [D,x]$ contains a core $\mathscr{X}$ for $D$. Further, suppose for all $h\in \mathscr{X},$ $$[D,x]h=yh$$ for some $y\in \mathscr{M}_D'.$ As $[D,x]$ is equal to a bounded operator on its domain of definition, which contains a core for $D$, Theorem\autoref{weaklyd} implies $x$ is weakly $D$-differentiable with $\delta^D_w(x)=iy.$ By Proposition\autoref{prop:ker=M_D'}, $iy\in \mathscr{M}_D'$ implies $iy\in \ker \delta^D_w.$ Thus, $x\in \ker (\delta^D_w)^2$ as $(\delta^D_w)^2(x)=\delta^D_w(iy)=0$. By Theorem\autoref{thm:kerD}, $x\in \ker(\delta^D_w)^2=\ker \delta^D_w$, so $iy=\delta^D_w(x)=0.$ Therefore, $y=0.$
\end{proof}

In particular, there is no $x\in B(H)$ so that $\dom [D,x]$ contains a core $\mathscr{X}$ for $D$ and $[D,x]h=ih$ for all $h\in \mathscr{X}$, i.e., no bounded operator can satisfy the Heisenberg Commutation Relation with a self-adjoint operator $D$ on a core for $D$.

\bibliography{References}

\begin{thebibliography}{10}

\bibitem{Bratteli-Robinson}
{\sc O.~Bratteli and D.~Robinson}, {\em Unbounded {Derivations} of
  {C}*-algebras}, Comm. Math. Phys., 42 (1975), pp.~253--268.

\bibitem{Bratteli-RobinsonQSM}
\leavevmode\vrule height 2pt depth -1.6pt width 23pt, {\em Operator Algebras
  and Quantum Statistical Mechanics: Volume 1}, Theoretical and Mathematical
  Physics, Springer Berlin Heidelberg, 2012.

\bibitem{Christensen2}
{\sc E.~Christensen}, {\em Higher {Weak} {Derivatives} and {Reflexive}
  {Algebras} of {Operators}}, in Operator Algebras and Their Applications,
  R.~Doran and E.~Park, eds., vol.~671 of Contemporary Mathematics, Province,
  RI, 2016, American Mathematical Society, pp.~69--83, arXiv:1504.03521v2
  [math.OA].

\bibitem{Christensen1}
\leavevmode\vrule height 2pt depth -1.6pt width 23pt, {\em On {Weakly}
  {$D$}-{Differentiable} {Operators}}, Expo. Math., 34 (2016), pp.~27--42,
  arXiv:1303.7426v4 [math.FA].

\bibitem{Conway}
{\sc J.~Conway}, {\em A Course in Functional Analysis}, Graduate Texts in
  Mathematics, Springer New York, 1994.

\bibitem{Dixmier}
{\sc J.~Dixmier}, {\em Sur la {Relation} $i({P}{Q}-{Q}{P})=1$}, Comp. Math., 13
  (1956--1958), pp.~263--269.

\bibitem{Kadison}
{\sc R.~Kadison}, {\em Derivations of {Operator} {Algebras}}, Ann. Math., 83
  (1966), pp.~280--293.

\bibitem{Kadison-Liu}
{\sc R.~Kadison and Z.~Liu}, {\em Derivations on {Murray}-von {Neumann}
  {Algebras}}, Math. Scand., 115 (2014), pp.~206--228.

\bibitem{Kaplansky}
{\sc I.~Kaplansky}, {\em Modules {Over} {Operator} {Algebras}}, Amer. J. Math.,
  75 (1953), pp.~839--858.

\bibitem{Sakai1}
{\sc S.~Sakai}, {\em On a {Conjecture} of {Kaplansky}}, Tohoku Math. J. (2), 12
  (1960), pp.~31--33.

\bibitem{Simon}
{\sc B.~Simon}, {\em Operator Theory}, A Comprehensive Course in Analysis--Part
  4, American Mathematical Society, 2015.

\bibitem{Singer}
{\sc I.~Singer and J.~Wermer}, {\em Derivations on {Commutative} {Normed}
  {Algebras}}, Ann. Math., 129 (1955), pp.~260--264.

\bibitem{Weidmann}
{\sc J.~Weidmann}, {\em Linear Operators in Hilbert Spaces}, Graduate Texts in
  Mathematics, Springer-Verlag, 1980.

\end{thebibliography}
\bibliographystyle{hsiam}

\end{document}